\documentclass[a4paper,12pt,draft]{amsart}
\usepackage[margin=30mm]{geometry}
\usepackage{amssymb, amsmath, amsthm, amscd}

\usepackage[T1]{fontenc}
\usepackage{eucal,mathrsfs,dsfont}%gives nicer set-names. Use \mathds instead of \mathds
\usepackage{color}%cyan,red;magenta,green;yellow,blue
\renewcommand{\leq}{\leqslant}
\renewcommand{\geq}{\geqslant}
\renewcommand{\le}{\leqslant}
\renewcommand{\ge}{\geqslant}

\definecolor{mno}{rgb}{0.5,0.1,0.5}

\newcommand{\R}{\mathds R}

\newcommand{\Pp}{\mathds P}
\newcommand{\Ee}{\mathds E}

\newcommand{\I}{\mathds 1}

\newcommand{\F}{\mathscr{F}}

\newtheorem{theorem}{Theorem}[section]
\newtheorem{lemma}[theorem]{Lemma}
\newtheorem{proposition}[theorem]{Proposition}
\newtheorem{corollary}[theorem]{Corollary}

\theoremstyle{definition}

\newtheorem{example}[theorem]{Example}
\newtheorem{remark}[theorem]{Remark}

\makeatletter % '@' is now a normal "letter" for TeX

\@addtoreset{equation}{section}
\makeatother % '@' is restored as a "non-letter" character for TeX

\begin{document}
\allowdisplaybreaks
\title{Littlewood--Paley--Stein Estimates for Non-local Dirichlet Forms}

\author{Huaiqian Li\qquad\quad Jian Wang}
\thanks{\emph{H.\ Li:}
Center for Applied Mathematics, Tianjin University, Tianjin 300072, P. R. China \texttt{huaiqianlee@gmail.com}}
 \thanks{\emph{J.\ Wang:}
 College of Mathematics and Informatics \& Fujian Key Laboratory of Mathematical Analysis and Applications (FJKLMAA), Fujian Normal University, 350007 Fuzhou, P. R. China. \texttt{jianwang@fjnu.edu.cn}}

%\date{}

\maketitle
\begin{abstract} We obtain the boundedness in $L^p$ spaces for all $1<p<\infty$ of the so-called vertical Littlewood--Paley functions for
non-local Dirichlet forms in the metric measure space under some mild assumptions. For $1<p\le 2$, the pseudo-gradient is
introduced to overcome the difficulty that chain rules are not
available for non-local operators, and then the Mosco convergence is used to pave the way from the finite jumping kernel case to the general case, while for $2\le p<\infty$, the Burkholder--Davis--Gundy inequality is effectively applied. The former
method is analytic and the latter one is probabilistic. The results extend those ones for pure jump symmetric L\'evy processes in Euclidean spaces.

\medskip

\medskip

\end{abstract}
\allowdisplaybreaks

\allowdisplaybreaks

\section{Introduction}\label{section1}
Let $(M,d)$ be a locally compact and separable metric space,  and
$\mu$ be a positive Radon measure on $M$ with full support. We will refer to such triple $(M,d,\mu)$ as a metric measure space. As usual, the real $L^p$ space is denoted by $L^p(M,\mu)$ with the norm
$$\|f\|_p:= \left(\int_M |f(x)|^p\, \mu(d x)\right)^{1/p},\quad1\leq p <\infty,$$
and
$$\|f\|_\infty := \textup{ess}\sup_{x\in M} |f(x)|,$$
where $\textup{ess}\sup$ is the essential supremum. The inner product of functions $f,g\in L^2(M,\mu)$ is denoted by $\langle f, g\rangle$.

Consider a Dirichlet form $(D,\mathscr{F})$ in $L^2(M,\mu)$, which is a closed, symmetric,
non-negative
definite, bilinear form $D: \mathscr{F}\times \mathscr{F}\rightarrow\R$ defined on a dense subspace $\mathscr{F}$ of $L^2(M,\mu)$, satisfying in addition the Markov property. The closedness means that $\mathscr{F}$ is a Hilbert space with respect to the
$D_1^{1/2}$-inner product defined by
$$D_1(f,g)=D(f,g)+\langle f, g\rangle.$$
The Markov property means that if $f\in\mathscr{F}$ then the function $\hat{f}:=\max\{0,\min\{1,f\}\}$ belongs to $\mathscr{F}$ and $D(\hat{f})\leq D(f)$. Here and in the sequel, we write $D(f)$ instead of $D(f,f)$ for short.

Let $L$ be the non-negative definite
$L^2$-generator
of the Dirichlet form $(D,\mathscr{F})$, which is a self-adjoint operator on $L^2(M,\mu)$ with domain $\mathscr{D}(L)$ such that
$$D(f,g)=\langle Lf,g\rangle,$$
for all $f\in\mathscr{D}(L)$ and $g\in\mathscr{F}$. The generator $L$ give rises to the semigroup $(P_t)_{t\geq0}$ with $P_t=e^{-tL}$ for all $t\geq0$ in the sense of functional calculus. It turns out that $(P_t)_{t\geq0}$ is a strongly continuous, contractive, symmetric semigroup in $L^2(M,\mu)$, and satisfies the Markov property which means that $0\leq P_tf\leq1$ for every $t>0$ provided $0\leq f\leq1$.

Let $C_c(M)$ be the space of all continuous functions on $M$ with compact support. Recall that the Dirichlet form  $(D,\mathscr{F})$ is
called regular if $\mathscr{F}\cap C_c(M)$ is dense both in $\mathscr{F}$ (with respect to the
$D_1^{1/2}$-norm) and in $C_c(M)$ (with respect to the supremum norm). It follows that if $(D,\mathscr{F})$ is regular, then every
function $f\in\mathscr{F}$ admits a quasi-continuous version $\tilde{f}$ (see e.g. \cite[Theorem 2.1.3]{FOT}). Throughout this
paper, we abuse the notation and
represent
$f\in \F$ by its quasi-continuous version without writing $\tilde f$.

In order to introduce the so-called vertical Littlewood--Paley square function, the ``module of gradient'' is necessary. The suitable candidate in this general setting should be the \emph{carr\'{e} du champ} operator. It is a non-negative, symmetric and continuous bilinear form $\Gamma: \mathscr{F}\times\mathscr{F}\rightarrow L^1(M,\mu)$ such that
$$D(f,g)=\int_M \Gamma(f,g)\,d\mu\quad\mbox{for every } f,g\in \mathscr{F},$$
which is uniquely characterized in the algebra $L^\infty(X,\mu)\cap \mathscr{F}$ by
$$\int_M\Gamma(f,g)h\,d\mu=D(f,gh)+D(g,fh)-D(fg,h),$$
for every $f,g,h\in L^\infty(X,\mu)\cap \mathscr{F}$. See \cite{BH1991} for more details. In the sequel, we use the notation $\Gamma(f):=\Gamma(f,f)$ for convenience.

\ \

In this paper, we are concerned with non-local Dirichlet forms. Let $(D,\mathscr{F})$ be a regular Dirichlet form of pure jump type in $L^2(M,\mu)$ defined as
\begin{equation}\label{nondi}
D(f,g)=\frac{1}{2}\iint_{M\times M \backslash {\rm diag}} (f(x)-f(y))(g(x)-g(y)) \,J(x,dy)\,\mu(dx),\quad f,g\in\mathscr{F},\end{equation}
where ${\rm diag}$ denotes the diagonal set $\{(x,x):x \in M\}$ and  $J(x,dy)$ is a non-negative kernel satisfying the symmetry property
 $$J(x,dy)\,\mu(dx)=J(y,dx)\,\mu(dy).$$
$J(x,dy)$ is called jumping kernel associated with the Dirichlet form $(D,\F)$ in the literature.
Then the\emph{ carr\'e du champ} operator $\Gamma$ is defined as follows
$$\Gamma(f,g)(x)=\frac{1}{2}\int_M (f(x)-f(y))(g(x)-g(y))\,J(x,dy),\quad f,g\in \mathscr{F} \text{ and } x\in M.$$
Clearly,
$$D(f)=\displaystyle\int_M \Gamma(f)(x)\,\mu(dx)\quad\mbox{for every }f\in\mathscr{F}.$$
This motivates us to define the gradient (more precisely, the module of  gradient) of a function $f\in \mathscr{F}$ by
\begin{equation}\label{g-1}
|\nabla f|(x)= \sqrt{ \Gamma(f)}(x)=
\left(\frac{1}{2}\int_M (f(x)-f(y))^2\,J(x,dy)\right)^{1/2},\quad x\in M.\end{equation}
Note that, due to the symmetry of $J(x,dy)\,\mu(dx)$, for every $f\in\mathscr{F}$,
$$D(f)= \iint_{\{(x,y)\in M\times M: f(x)\ge f(y)\}} (f(x)-f(y))^2 \,J(x,dy)\,\mu(dx).$$
Then, we can also well
define the following (module of) modified gradient
for every $f\in\mathscr{F}$,
\begin{equation}\label{g-2}
|\widetilde\nabla f|_*(x):=
\left(\int_{\{y\in M:\,f(x)\ge f(y)\}} (f(x)-f(y))^2\,J(x,dy)\right)^{1/2},\quad x\in M.\end{equation}
From \eqref{g-1} and \eqref{g-2}, it is easy to know that, for every $f\in\mathscr{F}$,
$$0\le |\widetilde\nabla f|_*\le \sqrt{2}|\nabla f|\quad\mbox{and}\quad\||\nabla f|\|_2^2= \||\widetilde\nabla f|_*\|_2^2=D(f,f).$$

Actually, motivated by \cite{BBL}, we need a further modification of the gradient (and this is a crucial point; see some remarks at the end of Section \ref{section2}). For any $f\in\mathscr{F}$, we define
\begin{equation}\label{g-21}
|\widetilde\nabla f|(x)=
\left(\int_{\{y\in M:\,|f|(x)\ge |f|(y)\}} (f(x)-f(y))^2\,J(x,dy)\right)^{1/2},\quad x\in M.\end{equation}
It is easy to see that
$|\widetilde\nabla f|=|\widetilde\nabla f|_*$ for any $0\le f\in \mathscr{F}$; however, for general $f\in \mathscr{F}$,
they are not comparable to each other. We also note that, similar to the standard module of gradient,
$|\widetilde\nabla f|= |\widetilde\nabla (-f)|$ for any $f\in \mathscr{F}$;
however, such property is not satisfied  for $|\widetilde\nabla \cdot|_*$. This in some sense indicates that the definition of the modified gradient $|\widetilde\nabla \cdot|$ above is more reasonable than that of $|\widetilde\nabla \cdot|_*$.

For every $f\in L^1(M,\mu)\cap L^\infty(M,\mu)$, we
now
define the vertical Littlewood--Paley $\mathscr{H}$-functions $\mathscr{H}_\nabla(f)$ and $\mathscr{H}_{\widetilde\nabla}(f)$ corresponding to the non-local Dirichlet form $(D,\mathscr{F})$ in \eqref{nondi} as
$$\mathscr{H}_\nabla(f)(x)=\left(\int_0^\infty |\nabla P_t f|^2(x)\,dt\right)^{1/2},$$
and
\begin{equation}\label{eeefff}\mathscr{H}_{\widetilde\nabla}(f)(x)=\left(\int_0^\infty |\widetilde\nabla P_tf|^2(x)\,dt\right)^{1/2},\end{equation}
for every $x\in M$.

The purpose of this paper is to establish Littlewood--Paley--Stein estimates in $L^p(M,\mu)$ for non-local Dirichlet form $(D,\mathscr{F})$
and for all $1<p<\infty$. The main result
is the following theorem (see Theorems \ref{th1} and \ref{thp} below for precise expressions).
\begin{theorem}\label{main}
Let $(M,d,\mu)$ be a metric measure space. Consider the non-local Dirichlet form $(D,\mathscr{F})$ defined in \eqref{nondi}. Under some
mild assumptions, for $p\in (1,2]$ the vertical Littlewood--Paley operator $\mathscr{H}_{\widetilde\nabla}$ is bounded in $L^p(M,\mu)$; for $p\in [2,\infty)$ the vertical Littlewood--Paley operator $\mathscr{H}_\nabla$ is bounded in $L^p(M,\mu)$.
\end{theorem}

The prototype of Littlewood--Paley--Stein estimates is the $L^p$ boundedness of the Littlewood--Paley $g$-function in the Euclidean space for all $1<p<\infty$;
see \cite[Chapter IV, Theorem 1]{St1970}. There are a lot of extensions on this result in various directions, and we only recall some of them.
We are interested in the vertical (i.e., derivative with respect to the spatial variable) Littlewood--Paley--Stein estimates for heat or Poisson semigroups. Let $M$ be a complete and connected (smooth) Riemannian manifold with Riemannian volume measure $dx$, the non-negative Laplace--Beltrami operator $\Delta$, the corresponding heat semigroup $(e^{-t\Delta})_{t\ge0}$ and Poisson semigroup $(e^{-t\sqrt{\Delta}})_{t\ge0}$, as well as the gradient operator $\nabla$. For every $f\in C_c^\infty(M)$, the vertical Littlewood--Paley $\mathscr{H}$- and $\mathscr{G}$-functions are given by
\begin{equation}\label{cla}
\mathscr{H}(f)(x)=\left(\int_0^\infty |\nabla e^{-t\Delta} f|^2(x)\,dt\right)^{1/2},
\end{equation}
 and \begin{equation}\label{cla-G}
\mathscr{G}(f)(x)=\left(\int_0^\infty t|\nabla e^{-t\sqrt{\Delta}} f|^2(x)\,dt\right)^{1/2},
\end{equation}
for every $x\in M$, where $|\cdot|$ is the length induced by the Riemannian distance in the tangent space. The operator $\mathscr{H}$
is called bounded in $L^p(M,dx)$ (or the Littlewood--Paley--Stein estimate holds for $\mathscr{H}$) for any $p\in (1,\infty)$, if there exists a constant $c_p>0$ such that
$$\|\mathscr{H}(f)\|_p\le c_p\|f\|_p,\quad f\in C_c^\infty(M).$$
(The same for $\mathscr{G}$.) On the aspect of analytic approaches, Stein \cite[Chapter II]{Stein} proved the $L^p$ boundedness of $\mathscr{G}$ for all $p\in (1,\infty)$ on compact Lie groups. Lohou\'{e} \cite{Lou1987} investigated the $L^p$ boundedness of the Littlewood--Paley $\mathscr{H}_a$- and $\mathscr{G}_a$-functions, defined as                                                                                                 $$\mathscr{H}_a(f)(x)=\Big(\int_0^\infty e^{at}|\nabla e^{-t\Delta}f(x)|^2\, dt\Big)^{1/2}$$
and $$\mathscr{G}_a(f)(x)=\Big(\int_0^\infty te^{at}|\nabla e^{-t\sqrt{\Delta}}f(x)|^2\, dt\Big)^{1/2}$$
in the Cartan--Hadamard manifold, where $a$ is a real number to be determined. In fact, no additional assumptions on $M$ are needed for the boundedness of $\mathscr{H}$ and $\mathscr{G}$ in $L^p(M,dx)$ for $1<p\leq2$ (see e.g. \cite{CDD}), while, for the case when $2<p<\infty$, much stronger assumptions are need (see e.g.\ \cite[Proposition 3.1]{CD2003}). On the aspect of probabilistic approaches, we should mention that Meyer \cite{Mey,Mey1981} studied the $L^p$ boundedness for all $1<p<\infty$ on the  Littlewood--Paley
$\mathscr{G}_*$-function, defined as
$$\mathscr{G}_*(f)(x)=\Big(\int_0^\infty te^{-2t\sqrt{\Delta}}|\nabla e^{-t\sqrt{\Delta}}f(x)|^2\, dt\Big)^{1/2}.$$
Bakry established a slightly different Littlewood--Paley--Stein estimate for diffusion processes under the condition that the Bakry--Emery $\Gamma_2$ is non-negative in \cite{Bakry1985}, and then proved it under the condition that $\Gamma_2$ is lower bounded on complete Riemannian manifolds in \cite{Bakry1987}. See also \cite{ShYo}, where strong assumptions are needed to guarantee a nice algebra and to run the $\Gamma$ calculus for diffusion processes. Li \cite{Li2006} established the Littlewood--Paley--Stein estimate for $\mathscr{G}$ on complete Riemannian manifolds, as well as the $L^p$ boundedness for $p\in(1,2]$ of Littlewood--Paley square functions for Poisson semigroups generated by the Hodge--Laplacian. We do not mention many studies on Wiener spaces here.

For non-local Dirichlet forms, to the authors'
knowledge, the study on the $L^p$ boundedness of the vertical Littlewood--Paley operator $\mathscr{H}$
is not too much. Dungey \cite{Nick} obtained the $L^p$ boundedness of the vertical Littlewood--Paley operator with $1<p\leq2$ for random walks on graphs and groups.  Ba\~{n}uelos, Bogdan and Luks \cite{BBL} studied Littlewood--Paley--Stein estimates for symmetric L\'evy processes in the Euclidean space recently (see \cite{BK} for more recent  extension on non-symmetric L\'evy processes). In the aforementioned papers on the L\'evy process case, the Euclidean structure is nice to apply the Hardy--Stein identity and it is used in a crucial way. However, the approach to prove our main result
(Theorem \ref{main} above), which is presented in the metric measure space setting, is different. Indeed, when $p\in (1,2]$, we prove the
boundedness of $\mathscr{H}_{\widetilde\nabla}$ by using the pseudo-gradient to overcome the difficulty that chain rules are not available
for non-local operators, and then by applying the Mosco convergence from finite jumping kernel case
to general case (see Theorem \ref{th1-23} below);  when $p\in [2,\infty)$, we verify the boundedness of $\mathscr{H}_\nabla$,
by following the idea of \cite{BBL} to express the square function as a conditional expectation of the quadratic variation of a suitable martingale and then applying the Burkholder--Davis--Gundy inequality (see Theorem \ref{thp} below). We would like to mention that, though $|\nabla f|$
seems more natural than $|\widetilde\nabla f|$, since $|\widetilde\nabla f|$ is of a certain asymmetry,  \cite[Example 2]{BBL} indicates that
in some settings $|\nabla f|$ may be too large to yield the boundedness of $\mathscr{H}_\nabla$ in $L^p(M,\mu)$ for $1<p< 2$.

To indicate clearly the contribution of our paper, we present two examples, for which Theorem \ref{main} (see also Theorems \ref{th1} and \ref{thp} below) is applicable.

\begin{example}\it Let $(D,\mathscr{F})$ be a regular Dirichlet form on $L^2(M,\mu)$ as follows
\begin{align*}D(f,f)=\frac{1}{2}\iint_{M\times M \backslash {\rm diag}}(f(y)-f(x))^2J(x,y)\,\mu(dx)\,\mu(dy),\quad f\in \F
\end{align*}
where $J(x,y)$ is a non-negative measurable symmetric function on $M\times M\backslash {\rm diag}$ such that
$$\int_{\{y\in M: d(x,y)>r\}} J(x,y)\,\mu(dy)<\infty,\quad x\in M, r>0.$$ Then, for any $p\in (1,2]$, the vertical Littlewood--Paley operator $\mathscr{H}_{\widetilde\nabla}$ is bounded in $L^p(M,\mu)$. \end{example}

The next example includes symmetric stable-like processes on $\R^d$ of variable orders.

\begin{example}\label{exm2}\it Let $s:\R^d\to [s_1,s_2]\subset (0,2)$ such that
$$|s(x)-s(y)|\le \frac{c}{\log(2/|x-y|)},\quad |x-y|\le 1$$ holds for some constant $c>0$.  Let $(D,\mathscr{F})$ be a regular Dirichlet form on $L^2(\R^d,dx)$ as follows
\begin{align*}D(f,f)=&\frac{1}{2}\iint(f(y)-f(x))^2j(x,y)\,dx\,dy,\\
\mathscr{F}=&\overline{ C_c^1(\R^d)}^{D^{1/2}_1},
\end{align*}
where $D_1(f,f)=D(f,f)+\|f\|_2^2$,  $j(x,y)$ is a non-negative symmetric measurable function on $\R^d\times \R^d$ such that
$$\sup_{x\in \R^d} \int_{\{|y-x|>1\}}j(x,y)\,dy<\infty,$$
and for some constants $c_1,c_2>0$
\begin{equation}\label{ed4}\frac{c_1}{|x-y|^{d+s(x)\wedge s(y)}}\le j(x,y)\le \frac{c_2}{|x-y|^{d+s(x)\vee s(y)}},\quad |x-y|\le 1.\end{equation} Then, for $p\in [2,\infty)$, the vertical Littlewood--Paley operator $\mathscr{H}_{\nabla}$ is bounded in $L^p(\R^d,dx)$. \end{example}

\ \

The next two sections are devoted to
proving the main result (Theorem \ref{main} above), which is treated separately according to $p\in (1,2]$ and $p\in [2,\infty)$.

\section{Littlewood--Paley--Stein estimates for $1<p\le 2$}\label{section2}
\subsection{Pseudo-gradient} In order to show the motivation, for the moment, let $M$ be a (smooth) Riemannian manifold, $\Delta$ be the Laplace--Beltrami operator, and $\nabla$ be the Riemannian gradient, and denote by $|\cdot|$ the length induced by the Riemannian distance in the tangent space.  A beautiful way to prove the $L^p$ boundedness of $\mathscr{H}$ and $\mathscr{G}$, defined by \eqref{cla-G} and \eqref{cla} respectively, is based on the following chain rule:
\begin{equation}\label{e:rie}
\Delta (f^p)=pf^{p-1} \Delta f-p(p-1)f^{p-2}|\nabla f|^2,
\end{equation}
which is valid for all $1<p<\infty$ and for all
 smooth functions $f$ on the Riemannian manifold $M$;
see \cite[Chapter IV, Lemma 1, p.\ 86]{Stein} or the proof of \cite[Lemma 2.1]{CDD} for example.

However, this so-called chain rule no longer holds for non-local
Dirichlet forms. For this,
following the idea of \cite{Nick},
we may make use of the following pseudo-gradient, which is defined by
\begin{equation}\label{e:ger}
\widetilde{\Gamma}_p(f)=pf(Lf)-f^{2-p}L(f^p),\end{equation}
for $p\in (1,\infty)$ and suitable non-negative
functions $f$,
where $L$ is the generator of the regular non-local Dirichlet form $(D,\mathscr{F})$ given in \eqref{nondi}. From \eqref{e:rie}, when $L$ is the Laplace--Beltrami operator $\Delta$ on the Riemannian manifold $M$, it is clear that the right hand side of \eqref{e:ger} is just $p(p-1)|\nabla f|^2$. Due to this, one can reasonably imagine that, in the general setting, $\widetilde{\Gamma}_p(f)$ should play the same role as $|\nabla f|^2$  in the Riemannian manifold setting. This is the reason why we call $\widetilde{\Gamma}_p(f)$ the  pseudo-gradient of $f$. For more details on the pseudo-gradient, refer to its origination \cite{Nick}.

For our purpose, we need to define the pseudo-gradient $\Gamma_p$ for suitable function $f$ (which is not necessarily non-negative)
as follows: for $p\in (1,2]$,
\begin{equation}\label{e:ger0}\Gamma_p(f)=pf(Lf)-|f|^{2-p}L(|f|^p).\end{equation} In particular, when $p=2$, $\Gamma_2(f)=2fLf-L(f^2)$. (In fact, for $p\in (2,\infty)$, we can also define $\Gamma_p(f)$ by the right hand side of \eqref{e:ger0} for suitable function $f$;
however, we will not use it in this work.) We emphasis that, to extend the definition of $\Gamma_p$ for signed function is one of the crucial points in our argument. This is a key difference between the discrete setting as in \cite{Nick} and the present setting for general  metric measure spaces.

Recall that $(M,d,\mu)$ is a metric measure space and $(D,\mathscr{F})$ is a non-local regular Dirichlet form of pure jump type defined in \eqref{nondi}. In the present setting, for a suitable function $f$ on $M$, $|\nabla f|$ and $|\widetilde\nabla f|$ are defined in \eqref{g-1} and \eqref{g-21}, respectively. In order to compare $\Gamma_p(f)$ with $|\nabla f|$ and $|\widetilde\nabla f|$, we need the closed expression of the generator $L$, which is difficult to seek in general (see e.g.\ \cite{SW2014}). However, if
\begin{equation}\label{e:bound} \int_{M\backslash\{x\} } J(x,dy)<\infty,\quad x\in M,\end{equation}
then it is easy to see that, for all $f\in \mathscr{D}(L)$,
 \begin{equation}\label{a12}Lf(x)=\int_M (f(x)-f(y))\,J(x,dy).\end{equation}
Note also that under \eqref{e:bound}, $Lf$ is pointwise $\mu$-a.e.\ well defined by \eqref{a12} for every $f\in L^\infty(M,\mu)$. We call that the jumping kernel $J(x,dy)$ is finite, if \eqref{e:bound} is satisfied.

\subsection{Boundedness of Littlewood--Paley functions for  $1<p\le  2$: the finite case}
Throughout this subsection, we always suppose that the jumping kernel $J(x,dy)$ is finite, i.e., \eqref{e:bound} holds.
The next lemma provides an explicit formula for $\Gamma_p(f)$ when $p\in (1,2]$ and $f\ge0$ (see \cite[Lemma 3.2]{Nick} for the case on graphs).
\begin{lemma}\label{Gamma_p}Under \eqref{e:bound}, for any $p\in (1,2]$ and $0\le f\in \mathscr{D}(L)\cap L^\infty(M,\mu)$,
it holds
$$\Gamma_p(f)(x)=p(p-1)\int_{\{y\in M:f(y)\neq f(x)\}}(f(x) -f(y))^2 \, I(f(x),f(y);p)\,J(x,dy),$$
for any $x\in M$, where
$$I(f(x),f(y);p)= \int_0^1  \frac{(1-u)f(x)^{2-p}}{((1-u)f(x)+uf(y))^{2-p}}\,du,$$
and $0^0:=1$.
\end{lemma}
\begin{remark}\label{remark333} On the one hand, Lemma \ref{Gamma_p} holds trivially when $p=2$; indeed, it is well known that, for every $f\in  \mathscr{D}(L)\cap L^\infty(M,\mu)$ and every $x\in M$,
$$\Gamma_2(f)(x)=\big[2f(Lf)-L(f^2)\big](x)= 2 \Gamma(f)(x)=\int_M (f(x)-f(y))^2\,J(x,dy).$$
On the other hand, we note that for $p\in (1,2)$ and $0\le f\in \mathscr{D}(L)\cap L^\infty(M,\mu)$ such that $f(x)\neq f(y)$ for $x,y\in M$, $$\int_0^1  \frac{(1-u)f(x)^{2-p}}{((1-u)f(x)+uf(y))^{2-p}}\,du \le \int_0^1  \frac{1}{(1-u)^{1-p}}\,du=\frac{1}{p},$$
and hence,  $\Gamma_p(f)\le (p-1) \Gamma_2(f)<\infty$. \end{remark}
\begin{proof}[Proof of Lemma $\ref{Gamma_p}$]
By the remark above, we only need to prove the case when $p\in (1,2)$.
According to \eqref{a12},  for $0\le f\in \mathscr{D}(L)\cap L^\infty(M,\mu)$,  we have
$$\Gamma_p(f)(x)=\int_M\big(pf(x)(f(x)-f(y))-f(x)^{2-p}(f(x)^p-f(y)^p)\big)\,J(x,dy).$$
Note that, to further calculate the right hand side of the equality above, we only need to consider the case when $f(y)\neq f(x)$ inside the integral.

By the Taylor expansion of the function $t\mapsto t^p$ on $[0,\infty)$, we have
\begin{align*}t^p-s^p=&ps^{p-1}(t-s)+p(p-1)\int_s^t v^{p-2}(t-v)\,dv\\
=&ps^{p-1}(t-s)+p(p-1)(t-s)^2\int_0^1 \frac{1-u}{((1-u)s+ut)^{2-p}}\,du\end{align*} for any $s,t\ge0$ with $s\neq t$, where the second equality follows by a change of variable, i.e., $v=(1-u)s+ut$. When $s=0$, the condition $p>1$ ensures that the integral  exists.
If $f(y)\neq f(x)$, then, setting $s=f(x)$ and $t=f(y)$, we obtain that
\begin{align*}&pf(x)(f(x)-f(y))-f(x)^{2-p}(f(x)^p-f(y)^p)\\
&=f(x)^{2-p}[(f(y)^p-f(x)^p)-pf(x)^{p-1} (f(y)-f(x))]\\
&=p(p-1)(f(x)-f(y))^2\int_0^1\frac{(1-u)f(x)^{2-p}}{((1-u)f(x)+uf(y))^{2-p}}\,du\\
&=p(p-1)(f(x)-f(y))^2 I(f(x),f(y);p).\end{align*} This yields
\begin{equation*}
\begin{split}\Gamma_p(f)(x)=&p(p-1)\int_{\{y\in M:f(y)\neq f(x)\}} (f(x)-f(y))^2 I(f(x),f(y);p)\,J(x,dy).\end{split}\end{equation*}
We prove the desired assertion. \end{proof}

Now we can immediately compare $\Gamma_p(f)$ with $|\nabla f|^2$ and $|\widetilde{\nabla} f|^2$ for suitable non-negative $f$.
\begin{corollary}\label{comp} Under \eqref{e:bound}, for $p\in (1,2]$ and $0\le f\in \mathscr{D}(L)\cap L^\infty(M,\mu)$,
$$0\le \Gamma_p(f)(x)\le 2(p-1)|\nabla f|^2(x)$$
and
\begin{equation}\label{e:ffcc}|\widetilde \nabla f|^2(x)\le 2/(p(p-1)) \Gamma_p(f)(x),\end{equation}
for any $x\in M$, where $|\nabla f|$ and $|\widetilde\nabla f|$ are defined by \eqref{g-1} and \eqref{g-21}, respectively.
 \end{corollary}
\begin{proof} Let $p\in (1,2)$ and
$0\le f\in \mathscr{D}(L)\cap L^\infty(M,\mu)$. We assume that \eqref{e:bound} holds.

(i) It is clear from Lemma \ref{Gamma_p} that $\Gamma_p(f)(x)\ge0$. Observing that
$$(1-u)^{2-p}f(x)^{2-p}\le ((1-u)f(x)+uf(y))^{2-p},$$ for any $0\leq u\leq1$ and $f\ge0$, we obtain that
\begin{align*}
\Gamma_p(f)(x)&\le p(p-1)\int_0^1(1-u)^{p-1} \, du\int_M (f(x)-f(y))^2  J(x,y)\,\mu(dy)\\
&= 2(p-1)|\nabla f|^2(x).
\end{align*}

(ii) Observing that, for $0\leq f(y)< f(x)$, one has
$(1-u)f(x)+uf(y)\le f(x)$ for any $0\leq u\leq1$. Hence
$$I(f(x),f(y);p)=\int_0^1 \frac{(1-u)f(x)^{2-p}}{((1-u)f(x)+uf(y))^{2-p}}\,du\ge \int_0^1 (1-u)\,du =1/2.$$ This along with the definition of $|\widetilde \nabla f|^2$ yields the desired assertion.
\end{proof}
\begin{remark} It is easy to see that, in general, $\Gamma_p(f)(x)< 2(p-1)|\nabla f|^2(x)$ holds  for any $0\le f\in \mathscr{D}(L)\cap L^\infty(M,\mu)$ under the assumption \eqref{e:bound}. For example,
 for any function $f$ with $f\neq 0$ and $f(x)=0$ for some $x\in M$, we have $|\nabla f|^2(x)>0$ and $\Gamma_p(f)(x)=0$. Therefore, for $p\in (1,2]$, in general situations, one can use the bound on $\Gamma_p(f)$ to control $|\widetilde \nabla f|^2$ but not $|\nabla f|^2$.
\end{remark}

The following statement shows that \eqref{e:ffcc} indeed holds for all $f\in \mathscr{D}(L)\cap L^\infty(M,\mu)$, which  is one of the key ingredients in our proof.

 \begin{proposition}\label{P:comp} Under \eqref{e:bound}, for $p\in (1,2]$ and $f\in \mathscr{D}(L)\cap L^\infty(M,\mu)$,
$$0\le |\widetilde \nabla f|^2(x)\le 2/(p(p-1)) \Gamma_p(f)(x),$$
for any $x\in M$, where $|\widetilde\nabla f|$ is defined by \eqref{g-21}.
 \end{proposition}
 \begin{proof} By Remark \ref{remark333} again, without loss of generality we may and can
  assume that $p\in (1,2)$. The proof is a little bit delicate and is based on Corollary \ref{comp}.
 For any $f\in \mathscr{D}(L)\cap L^\infty(M,\mu)$,
by \eqref{e:ger0},
\begin{align*}
\Gamma_p(f)&=pfLf-p|f|(L|f|)+p|f|L|f|-|f|^{2-p}L(|f|^p)\\
&=pfLf-p|f|L|f|+\Gamma_p(|f|).
\end{align*}
According to \eqref{e:ffcc} in Corollary \ref{comp}, it holds that
\begin{align*}\Gamma_p(|f|)(x)\ge &\frac{p(p-1)}{2}\big|\widetilde \nabla |f|\big|^2(x)\\
\ge &\frac{p(p-1)}{2}\int_{\{y\in M:|f|(x)\ge |f|(y)\}}(|f|(x)-|f|(y))^2\,J(x,dy).
\end{align*}

On the other hand,
\begin{align*}&pf(x)(Lf)(x)-p|f|(x)(L|f|)(x)\\
&=p\left(\int_M f(x)(f(x)-f(y))\,J(x,dy)-\int_M |f|(x)(|f|(x)-|f|(y))\,J(x,dy)\right)\\
&=p\int_M (|f|(x)|f|(y)-f(x)f(y))\,J(x,dy)\\
&\ge p\int_{\{y\in M:|f|(x)\ge |f|(y)\}}(|f|(x)|f|(y)-f(x)f(y))\,J(x,dy),  \end{align*}  where in the last inequality we used the fact that
$|f|(x)|f|(y)-f(x)f(y)\ge0$ for all $x,y\in M$.

Furthermore, we deduce that
\begin{align*}&\frac{p-1}{2}(|f|(x)-|f|(y))^2+|f|(x)|f|(y)-f(x)f(y)\\
&=\frac{p-1}{2}(f^2(x)+f^2(y))+(2-p)|f|(x)|f|(y)-f(x)f(y)\\
&\ge \frac{p-1}{2}(f^2(x)+f^2(y))+(1-p)|f|(x)|f|(y)\\
&\ge \frac{p-1}{2}(f^2(x)+f^2(y))-(p-1)f(x)f(y)\\
&=\frac{p-1}{2}(f(x)-f(y))^2,\end{align*}
where both inequalities follow from the fact that
$|f|(x)|f|(y)-f(x)f(y)\ge0$ for all $x,y\in M$ again.

Combining with all the inequalities above, we arrive at the desired assertion.
 \end{proof}

Recall that $(P_t)_{t\geq0}$ with $P_t=e^{-tL}$ is the semigroup corresponding to the Dirichlet form $(D,\mathscr{F})$.
\begin{proposition}\label{pre}\,\ Let $(M,d,\mu)$ be a metric measure space, and $(D,\mathscr{F})$ be the Dirichlet form defined in \eqref{nondi}.
Suppose that  \eqref{e:bound} holds. Then, for any $p\in (1,2]$, the following assertions hold.
\begin{itemize}
\item[{\rm(i)}]  There exists a constant $c_p>0$ such that, for all $t>0$ and $f\in L^1(M,\mu)\cap L^\infty(M,\mu)$,
$$\|\Gamma_p^{1/2}(P_tf)\|_p\le c_p t^{-1/2} \|f\|_p.$$

\item[{\rm(ii)}]  For any $f\in L^1(M,\mu)\cap L^\infty(M,\mu)$, define
$$(\mathscr{H}_pf)(x)=\left(\int_0^\infty\Gamma_p (P_tf)(x)\,dt\right)^{1/2}\quad\mbox{for every } x\in M.$$
Then there is a constant $c_p'>0$ such that, for all $f\in L^1(M,\mu)\cap L^\infty(M,\mu)$,
$$\|\mathscr{H}_pf\|_p\le c'_p\|f\|_p.$$
\end{itemize}
\end{proposition}

\begin{proof}
Noticing that $(P_t)_{t\ge0}$ is a strongly continuous Markovian semigroup defined on $L^2(M,\mu)$, the operator $P_t$ can be extended to $L^\infty(M,\mu)$ such that $\|P_tf\|_{\infty}\le \|f\|_{\infty}$ (see \cite[Page\ 56]{FOT}). On the other hand, we can also extend $P_t$ on $L^1(M,\mu)\cap L^2(M,\mu)$ to $L^1(M,\mu)$ uniquely. Since $(P_t)_{t\ge0}$ is a symmetric Markovian semigroup, it holds that $\|P_tf\|_{1}\le \|f\|_{1}$. By the Riesz--Thorin interpolation theorem, for all $p\in(1,\infty)$, we have $$\|P_t\|_{p\to p}:=\sup_{f\in L^p(M,\mu)\setminus\{0\}}\frac{\|P_tf\|_{p}}{\|f\|_{p}}\le 1.$$

(i) In the following, let $p\in (1,2]$, and consider any non-zero function $f\in L^1(M,\mu)\cap L^\infty(M,\mu)$. Set $u_t=P_tf$ for all $t\ge 0$. Then $u_t\in \mathscr{D}(L)\cap L^1(M,\mu)\cap L^\infty(M,\mu)$ for every $t>0$.
In what follows, let $\partial_t$ denote the differentiation with respect to $t$.
The fundamental idea of the proof below is due to \cite{Stein}; however, we may not reduce the problem to take non-negative function $f$ as in the aforementioned reference, since $|\widetilde\nabla \cdot|$ does not enjoy the sublinear property.
 Instead, we should take into account $|u_t|$ not $u_t$ itself, which also explains the reason why we need to define the pseudo-gradient $\Gamma_p$ for all suitable signed functions by \eqref{e:ger0}.
 By the definition of $\Gamma_p$ and the fact that \begin{equation}\label{eq-t}\partial_t(|u_t|^p)=pu_t|u_t|^{p-2}\partial_tu_t=-pu_t|u_t|^{p-2}Lu_t,\end{equation}
we have
\begin{align*}|u_t|^{p-2}\Gamma_p(u_t)&=|u_t|^{p-2}\left(pu_t Lu_t-|u_t|^{2-p}L|u_t|^p\right)\\
&=pu_t |u_t|^{p-2}(Lu_t)-L(|u_t|^p)\\
&=-(\partial_t+L)(|u_t|^p).\end{align*}
It follows that
$$\Gamma_p(u_t)=-|u_t|^{2-p}(\partial_t+L)(|u_t|^p).$$
Set
$$J_t:=-(\partial_t+L)(|u_t|^p).$$
Then $J_t\geq0$ since $\Gamma_p(u_t)\geq0$ by Lemma \ref{Gamma_p}. Using the H\"{o}lder inequality, we have
\begin{align*} \|\Gamma_p^{1/2}(u_t)\|_p^p&=\int_M \Gamma_p^{p/2}(u_t)(x)\,\mu(dx)=\int_M |u_t|^{p(2-p)/2}(x) J_t^{p/2}(x)\,\mu(dx)\\
&\le \left(\int_M J_t(x)\,\mu(dx)\right)^{p/2}\left(\int_M |u_t|^p(x)\,\mu(dx)\right)^{(2-p)/2}.\end{align*}

Note that, by the
contraction
property of  the semigroup $(P_t)_{t\ge0}$ on $L^p(M,\mu)$,
$$\left(\int_M |u_t|^p(x)\,\mu(dx)\right)^{(2-p)/2}\le \|f\|_p^{p(2-p)/2}.$$
On the other hand, by H\"{o}lder's inequality and the contraction property again,
\begin{align*}\int_M |\partial _t|u_t|^p|(x)\,\mu(dx)=&p\int_M |u_t|^{p-1}| |\partial_tu_t|\,d\mu \le p\|\partial_tu\|_p\||u_t|^{p-1}\|_{p/(p-1)}\\
\le& p\|\partial_tu_t\|_p\|f\|_p^{p-1}.
\end{align*}

Recall that $L$ is self-adjoint in $L^2(M,\mu)$ and $P_t$ is continuous as a map from $L^p(M,\mu)$ to itself for every $t\ge0$ and for every $p\in[1,\infty]$. By the classical theory developed by Stein (see e.g.  \cite[Chapter III, Theorem 1]{Stein}), we derive that $(P_t)_{t>0}$ is an analytic semigroup in $L^p(M,\mu)$ for every $p\in(1,\infty)$; more precisely, the map $t\mapsto P_t$ has an analytic extension in the sense that it extends to an analytic $L^p(M, \mu)$-operator-valued function $t + is\mapsto P_{t+is}=e^{-(t+is)L}$, which is defined in the sector of the complex plane
$$|\arg(t+is)|<\frac{\pi}{2}\Big(1-\Big|\frac{2}{p}-1\Big|\Big).$$
Hence, we find that $$\|\partial_tu_t\|_p=\|Lu_t\|_p\le c_p t^{-1}\|f\|_p.$$ Thus, together with \eqref{eq-t},
\begin{equation}\label{eeff} \int_M |\partial _t|u_t|^p|(x)\,\mu(dx)\le pc_p t^{-1}\|f\|_p^{p}.\end{equation}
In particular, due to $J_t\ge0$ and
$$L|u_t|^p=-\partial_t|u_t|^p-J_t,$$
we get that
\begin{equation}\label{eeff0} \mu((L|u_t|^p)^+)<\infty.\end{equation}

Thus, \begin{equation}\label{eeff1}\int_M L|u_t|^p(x)\,\mu(dx)\ge0.\end{equation}
Indeed, let $\{K_n\}_{n\ge1}$  be a sequence of increasing compact sets such that $\cup_{n=1}^\infty K_n=M$, and
let $\{\varphi_n\}_{n\ge1}$ be a sequence of bounded measurable functions such that $\varphi_n=1$ on $K_n$, and $0\le \varphi_n\le 1$  on $K_n^c$.
Using \eqref{eeff0} and the extension of Fatou's lemma (see \cite[Theorem 3.2.6 (2), p.\ 52]{YJA}), we get
\begin{equation*}\begin{split}
\int_M L|u_t|^p(x)\,\mu(dx)&\ge\limsup_{n\to \infty} \int_M \varphi_n(x) (L|u_t|^p)(x)\,\mu(dx)\\
&= \limsup_{n\to \infty} \int_M |u_t|^p(x) (L \varphi_n)(x)\,\mu(dx),
\end{split} \end{equation*}
where in the equality above we used the symmetry property of $J(x,dy)\,\mu(dx)$ and the facts that $Lf$ is pointwise $\mu$-a.e. well defined for any bounded measurable function $f$ and $|u_t|^p\in L^1(M,\mu)$.
On the other hand, since $(D,\F)$ is a regular Dirichlet form on $L^2(M,\mu)$, for any relatively compact open sets $U$ and $V$ with $\bar{U}\subset V$, there is a function $\psi\in \F\cap C_c(M)$ such that $\psi=1$ on $U$ and $\psi=0$ on $V^c$. Consequently,
\begin{equation}\label{e:ffee}\begin{split}\iint_{U\times V^c} J(x,dy)\,\mu(dx)=&\iint_{U\times V^c}(\psi(x)-\psi(y))^2J(x,dy)\,\mu(dx)\\
\le &D(\psi,\psi)<\infty.\end{split}\end{equation}
For any fixed $x\in M$ and any $\varepsilon>0$, by \eqref{e:bound}, we can choose $R:=R(x,\varepsilon)>0$ large enough such that
$$\displaystyle\int_{\{y\in M: d(x,y)\ge R\}}\,J(x,dy)<\varepsilon.$$
 Fix this $R$. Then, for $n\ge 1$ large enough, $\varphi_n(y)=1$ for all $y\in M$ with $d(x,y)< R$. Thus, for $n\ge 1$ large enough,
$$|L\varphi_n(x)|=\bigg|\int_M\left(\varphi_n(x)-\varphi_n(y)\right)\, J(x,d y) \bigg|\le \int_{\{y\in M: d(x,y)\ge R\}}\,J(x,dy)<\varepsilon,$$
which means that $\lim_{n\to\infty }L\varphi_n(x)=0$ for all $x\in M$.
This, along with \eqref{e:ffee}, the fact that $|u_t|^p\in L^1(M,\mu)\cap L^\infty(M,\mu)$ and the dominated convergence theorem, gives
$$\limsup_{n\to \infty} \int_M |u_t|^p(x) (L \varphi_n)(x)\,\mu(dx)=0.$$ So, \eqref{eeff1} holds
true.

\eqref{eeff1} together with \eqref{eeff} yields that
\begin{align*}\int_M J_t(x)\,\mu(dx)&=-\int_M\partial _t|u_t|^p(x)\,\mu(dx)-\int_M L|u_t|^p(x)\,\mu(dx)\le pc_p t^{-1}\|f\|_p^{p}.\end{align*}
Hence,
$$ \left(\int_M J_t(x)\,\mu(dx)\right)^{p/2}\le c'_p t^{-p/2}\|f\|_p^{p^2/2}.$$

Combining with all the conclusions above, we prove the first assertion.

(ii) Note that
\begin{align*}(\mathscr{H}_pf)^2(x)&=\int_0^\infty \Gamma_p(u_t)(x)\,dt\\
&=-\int_0^\infty |u_t|(x)^{2-p}(\partial_t+L)(|u_t|^p)(x)\,dt\\
&\le f^*(x)^{2-p} J(x),
\end{align*}
where $f^*$ is the semigroup maximal function defined by \eqref{maxf} below, and
$$J(x)=-\int_0^\infty (\partial_t+L)(|u_t|^p)(x)\,dt,$$
which is non-negative. Thus, by using the H\"{o}lder inequality,
\begin{align*}\int_M( \mathscr{H}_p f)^p(x)\,\mu(dx)\le&\int_M f^*(x)^{p(2-p)/2}J(x)^{p/2}\,\mu(dx)\\
\le &\left(\int_M f^*(x)^p\,\mu(dx)\right)^{(2-p)/2}\left(\int_M J(x)\,\mu(dx)\right)^{p/2}.\end{align*}
 Lemma \ref{max} below further yields that
 $$\left(\int_M f^*(x)^p\,\mu(dx)\right)^{(2-p)/2}\le c_p'\|f\|_p^{p(2-p)/2}.$$
 On the other hand, by \eqref{eeff1},
 \begin{align*}\int_M J(x)\,\mu(dx) =&-\int_0^\infty dt \int_M (\partial_t+L) |u_t|^p(x)\,\mu(dx)\\
 \le&-\int_0^\infty dt \int_M \partial_t |u_t|^p(x)\,\mu(dx)\\
 \le& \int_M |f|^p(x)\,\mu(dx)=\|f\|_p^p.\end{align*}
 Combining all the inequalities above, we obtain that
 $$\int_M( \mathscr{H}_p f)^p(x)\,\mu(dx)\le c_p'\|f\|_p^p,$$
 which is the second assertion.

 Therefore, the proof is complete.
 \end{proof}

For any $f\in L^1(M,\mu)\cap L^\infty(M,\mu)$, define the semigroup maximal function $f^*$ by
\begin{equation}\label{maxf}f^*(x)=\sup_{t>0} |P_tf(x)|,\quad x\in M.\end{equation}
Since $(P_t)_{t\ge0}$ is a symmetric sub-Markovian semigroup, we have the following
\begin{lemma}\label{max}$($\cite[Section III. 3, p.\ 73]{Stein}$)$
For all $p\in (1,\infty]$, there exists a constant $c_p>0$ such that for all $f\in L^p(M,\mu)$,
$$\|f^*\|_p\le c_p\|f\|_p,$$
where, for $p=\infty$, the right hand side is just $\|f\|_{\infty}$ $($i.e., the constant $c_\infty=1)$.
\end{lemma}

Furthermore, according to Proposition \ref{P:comp} and  Proposition \ref{pre},  
we immediately have the following
\begin{theorem}\label{th1} {\bf (Finite jumping kernel case)}\,\,
Under the same assumption of Proposition $\ref{pre}$, for any $p\in (1,2]$,
 $\mathscr{H}_{\widetilde \nabla}$ is bounded in $L^p(M,\mu)$, i.e., there exists a constant $c_p>0$ such that, for every $f\in L^p(M,\mu)$,
 $$\|\mathscr{H}_{\widetilde \nabla}f\|_{p}\le c_p\|f\|_{p}.$$
Moreover, there exists a constant $\tilde{c}_p>0$ such that, for every $f\in L^p(M,\mu)$,
 $$\||\widetilde \nabla P_t f|\|_{p}\le \tilde{c}_p t^{-1/2}\|f\|_{p}.$$ \end{theorem}
\begin{proof}
For any $f\in L^1(M,
\mu)\cap L^\infty(M,\mu)$, $P_tf$  belongs to $\mathscr{D}(L)\cap L^1(M,\mu)\cap L^\infty(M,\mu)$ for every $t>0$. By Proposition \ref{P:comp}, we deduce that
\begin{align*}
\mathscr{H}_{\widetilde \nabla} f(x)&=\left(\int_0^\infty |\widetilde \nabla P_tf|^2(x)\,dt \right)^{1/2}\\
&\leq\left(\int_0^\infty\Big|\frac{2}{p(p-1)}\Gamma_p(P_tf)(x)\Big|\,d t\right)^{1/2}\\
&=\Big(\frac{2}{p(p-1)}\Big)^{1/2}(\mathscr{H}_pf)(x).
\end{align*}
By Proposition \ref{pre}(ii), we have
$$\|\mathscr{H}_{\widetilde \nabla} f\|_p\leq c_p'\Big(\frac{2}{p(p-1)}\Big)^{1/2}\|f\|_p=:c_p\|f\|_p.$$
Now the general case when $f\in L^p(M,\mu)$ follows from the fact that $L^1(M,\mu)\cap L^\infty(M,\mu)$ is dense in $L^p(M,\mu)$ and the application of Fatou's lemma.

The last assertion is also an immediate application of Proposition \ref{P:comp} and Proposition \ref{pre}(i).
\end{proof}
\subsection{Boundedness of Littlewood--Paley functions for $1<p\le 2$: the general case}
In this part, we consider the case that \eqref{e:bound} is not necessarily satisfied. In this general setting,
it is not clear that $Lf^p$ and so $\Gamma_p(f)$ given by \eqref{e:ger} are well defined for $p\in(1,2]$
and $0\le f\in \mathscr{D}(L)\cap L^\infty(M,\mu)$ (even in the pointwise sense).  (Note that, according to \cite[Theorem 2]{Mey},
if $f\in \mathscr{D}(L)\cap L^\infty(M,\mu)$, then $f^p\in \mathscr{D}(L)\cap L^\infty(M,\mu)$ for all $p\in[2,\infty)$.) To overcome this difficulty, we will make use of the Mosco convergence of non-local Dirichlet forms, and impose the absolute continuity and
the local finiteness assumptions on the jumping kernel $J(x,dy)$, i.e., there is a non-negative measurable function $J(x,y)$ on $M\times M \setminus {\rm diag}$ such that, for every $x,y\in M$,
\begin{equation}\label{e:ack}
J(x,dy)=J(x,y)\,\mu(dy),
\end{equation}
and
\begin{equation}
\label{e:ack1}\int_{\{y\in M: d(x,y)>r\}}J(x,y)\,\mu(dy)<\infty,\quad x\in M,\, r>0.
\end{equation}

\ \

Recall that, a sequence of Dirichlet forms $\{(D^n,\F^n)\}_{n\ge1}$ on $L^2(M,\mu)$ is said to be convergent to a Dirichlet
form $(D,\F)$ in $L^2(M,\mu)$ in the sense of Mosco if
\begin{itemize}
\item[(a)] for every sequence $\{f_n\}_{n\ge1}$ in $L^2(M,\mu)$ converging weakly to $f$ in $L^2(M,\mu)$,
$$\liminf_{n\to \infty} D^n(f_n,f_n)\ge D(f,f);$$
\item[(b)] for every $f\in L^2(M,\mu)$, there is a sequence $\{f_n\}_{n\ge1}$ in $L^2(M,\mu)$ converging strongly to $f$ in $L^2(M,\mu)$ such that
    $$\limsup_{n\to \infty} D^n(f_n,f_n)\le D(f,f).$$
\end{itemize}

\begin{remark}\label{r:mos} We make the following two comments on Mosco convergence.
\begin{itemize}
\item[(1)] Condition (b) in the definition of Mosco convergence is implied by the following condition:
\begin{itemize}
\item[(b)'] There is a common core $\mathscr{C}$ for the Dirichlet forms $\{(D^n,\F^n)\}_{n\ge1}$ and $(D,\F)$ such that
$$\lim_{n\to\infty} D^n(f,f)=D(f,f),\quad\mbox{for every } f\in \mathscr{C}.$$ See the proof of \cite[Thoorem 2.3]{BBCK}.
\end{itemize}
\item[(2)] Let $\{(D^n,\F^n)\}_{n\ge1}$ and $(D,\F)$ be Dirichlet
forms on $L^2(M,\mu)$. Then, the sequence $\{(D^n,\F^n)\}_{n\ge1}$ converges to $(D,\F)$ in the sense of Mosco, if and only if, for every $t>0$ and $f\in L^2(M,\mu)$, $P^{(n)}_tf$ converges to $P_tf$ in $L^2(M,\mu)$, where $(P_t)_{t\ge0}$ and $(P^{(n)}_t)_{t\ge0}$ are the semigroups corresponding to $(D,\F)$ and $(D^n,\F^n)$, respectively. See \cite[Corollary 2.6.1]{Mo1994}.
\end{itemize}
\end{remark}

\ \

Now, we consider the regular non-local Dirichlet form $(D,\F)$ given by \eqref{nondi}, and suppose that \eqref{e:ack} is satisfied.
For any $n\ge 1$ and $x,y\in M$, define
$$J_n(x,y)=J(x,y)\I_{\{d(x,y)>1/n\}}.$$
Note that \eqref{e:ffee} holds for general regular Dirichlet form $(D,\F)$.
 Then, by the definition of $J_n(x,y)$ and \eqref{e:ffee}, the sequence $\{J_n(x,y)\}_{n\ge1}$ converges to $J(x,y)$ locally in $L^1(M\times M\backslash {\rm diag}, \mu\times \mu)$.

Let $\mathscr{C}:=\F\cap C_c(M)$ be a core of $(D,\F)$.
For any $n\ge1$, we define the regular Dirichlet form $(D^n,\F^n)$ as follows
\begin{align*}
D^n(f,f)=&\iint (f(x)-f(y))^2 J_n(x,y)\,\mu(dx)\,\mu(dy),\\
\F^n=& \overline{\mathscr{C}}^{\sqrt{D^{n,1}}},\end{align*} where ${D^{n,1}}(f,f)=D^n(f,f)+\|f\|_2^2$
and $\overline{\mathscr{C}}^{\sqrt{D^{n,1}}}$ denotes the closure of $\mathscr{C}$ with respect to the metric $\sqrt{D^{n,1}}$.
Note that $J_n(x,y)\le J(x,y)$, $\F\subset \F^n$ for all $n\ge 1$. In particular, $\mathscr{C}$ is a common core for all $(D^n,\F^n)$, $n\geq1$. Furthermore, we have the following statement.

\begin{proposition}\label{P:mos} Under \eqref{e:ack}, the sequence of Dirichlet forms $\{(D^n,\F^n)\}_{n\ge1}$ above converges to $(D,\F)$ in the sense of Mosco. \end{proposition}
\begin{proof} We split the proof into two parts.

(1) In this part, the argument is inspired by the proof of \cite[Theorem 1.4]{SU}. Suppose that  $u_n$ is weakly convergent to $u$ in $L^2(M,\mu)$   as $n\rightarrow\infty$, and $$\liminf_{n\to \infty} \iint (u_n(x)-u_n(y))^2 J_n(x,y)\,\mu(dx)\,\mu(dy)<\infty.$$ We may assume that $$\lim_{n\to \infty} \iint (u_n(x)-u_n(y))^2 J_n(x,y)\,\mu(dx)\,\mu(dy)<\infty.$$ For $(x,y)\in M\times M \backslash {\rm diag}$ and $n\ge 1$, define $$\tilde u_n(x,y)=(u_n(x)-u_n(y))J_n(x,y)^{1/2}.$$ Then $\{\tilde u_n\}_{n\ge1}$ is a bounded sequence in $L^2(M\times M \backslash {\rm diag}, \mu\times \mu)$, and hence there exists a subsequence $\{\tilde u_{n_k}\}_{k\ge1}$,  which converges to some element $\tilde u$ weakly  in $L^2 (M\times M \backslash {\rm diag}, \mu\times \mu)$. We now claim that
\begin{equation}\label{e:ff0}\tilde u(x,y)= (u(x)-u(y))J(x,y)^{1/2},\quad \mu\times \mu\mbox{-}a.e.\,\ (x,y) \text{ with }x\neq y.\end{equation}

To simplify the notation, without confusion in double integrals below we will omit the integral domain $M\times M\backslash {\rm diag}$. For any non-negative function $v\in C_c(M\times M \backslash{\rm diag})$  and for any $n_k$, we have
\begin{equation*}\begin{split}
&\bigg|\iint \big[\tilde u(x,y)-(u(x)-u(y)) J(x,y)^{1/2} \big] v(x,y)\,\mu(dx)\,\mu(dy)\bigg|\\
&\le \bigg|\iint\big[\tilde u(x,y)-\big(u_{n_k}(x)-u_{n_k}(y)\big) J_{n_k}(x,y)^{1/2}\big]  v(x,y)\,\mu(dx)\,\mu(dy)\bigg|\\
&\quad + \bigg|\iint \big(u_{n_k}(x)-u_{n_k}(y)\big)\big(J_{n_k}(x,y)^{1/2}-J(x,y)^{1/2}\big)v(x,y)\,\mu(dx)\,\mu(dy)\bigg|\\
&\quad +  \bigg|\iint \big[\big(u_{n_k}(x)-u_{n_k}(y)\big)-\big(u(x)-u(y)\big)\big]J(x,y)^{1/2}v(x,y)\,\mu(dx)\,\mu(dy)\bigg|\\
&=:I_{1,n_k}+I_{2,n_k}+I_{3,n_k}.\end{split} \end{equation*}

Firstly, since $\tilde u_n$ converges to $\tilde u$ weakly in $L^2(M\times M \backslash {\rm diag}, \mu\times \mu)$, we see that $\lim_{k\to \infty}I_{1,n_k}=0.$
Secondly, by using the
Cauchy--Schwarz inequality and the fact that $\{u_{n_k}\}_{k\ge1} $ is a bounded sequence in $L^2(M,\mu)$, we derive that
\begin{align*}I_{2,n_k}\le &\left( \iint\big(u_{n_k}(x)-u_{n_k}(y)\big)^2 v(x,y)\,\mu(dx)\,\mu(dy)\right)^{1/2}\\
&\times \left(\iint\big(J_{n_k}(x,y)^{1/2}-J(x,y)^{1/2}\big)^2 v(x,y)\,\mu(dx)\,\mu(dy)\right)^{1/2}\\
\le& \sqrt{2} \|u_{n_k}\|_2     \|v\|_\infty  \\
&\times
 \left(\sup_{x\in M} \int_{\{y\in M: (x,y)\in {\rm supp} v\}} \,\mu(dy)+\sup_{y\in M} \int_{\{x\in M: (x,y)\in {\rm supp} v\}} \,\mu(dx)\right)^{1/2}\\
 &\times \left(\iint_{{\rm supp}v} |J_{n_k}(x,y)-J(x,y)|\,\mu(dx)\,\mu(dy)\right)^{1/2},
\end{align*}
where the right hand side of the above inequality converges to 0  as $k\to \infty$. Here, in the second inequality above, we used the elementary inequalities $(a-b)^2\le 2(a^2+b^2)$ for all $a,b\in \R$, and $|\sqrt{a}-\sqrt{b}|\le \sqrt{|a-b|}$ for all $a,b\ge 0$, and in the last inequality, we used the fact that
 $$\sup_{x\in M} \int_{\{y\in M: (x,y)\in {\rm supp} v\}} \,\mu(dy)+\sup_{y\in M} \int_{\{x\in M: (x,y)\in {\rm supp} v\}} \,\mu(dx)<\infty,$$ due to $v\in C_c(M\times M \backslash{\rm diag})$.
Thirdly, for $I_{3,n_k}$, note that, by the Cauchy--Schwarz inequality and \eqref{e:ffee}, both
 $$\phi(x):=\int_M J(x,y)^{1/2} v(x,y)\,\mu(dy),\quad x\in M$$
 and
 $$\psi(y):=\int_M J(x,y)^{1/2} v(x,y)\,\mu(dx),\quad y\in M$$
 are in $L^2(M,\mu)$.
 Hence, we see
 \begin{align*}I_{3,n_k}\le & \bigg| \int_M \big(u_{n_k}(x)-u(x)\big)\phi(x)\,\mu(dx)\bigg|+ \bigg| \int_M \big(u_{n_k}(y)-u(y)\big)\psi(y)\,\mu(dy)\bigg|\end{align*} goes to 0 as $k\to \infty$. Thus, we conclude that \eqref{e:ff0} holds.

Choose a sequence  $\{v_k\}_{k\geq1}$ from $C_c(M\times M \backslash{\rm diag})$ such that $0\le v_k\uparrow 1$ as $k\to \infty$. Letting $n\rightarrow\infty$, by Fatou's lemma, we deduce that, for any $k\ge 1$,
\begin{align*}
\liminf_{n\to \infty} D^n(u_n,u_n)\ge &\liminf_{n\to \infty} \iint \big(u_n(x)-u_n(y)\big)^2 J_n(x,y) v_k(x,y)\,\mu(dx)\,\mu(dy)\\
\ge &\iint \big(u(x)-u(y)\big)^2 J(x,y) v_k(x,y)\,\mu(dx)\,\mu(dy).
\end{align*}
Taking $k\to \infty$, by the monotone convergence theorem, we arrive at
$$  \liminf_{n\to \infty} D^n(u_n,u_n)\ge D(u,u).$$

(2) Note that $\mathscr{C}$ is the common core of $(D^n,\F^n)$ for all $n\ge 1$ and $(D,\F)$. Hence, by the monotone convergence theorem,
$$\lim_{n\to \infty} D^n(u,u)= D(u,u),\quad u\in \mathscr{C}.$$
This proves that the condition $(b)'$ in Remark \ref{r:mos} (1) holds.

Combining both conclusions from (1) and (2), we prove the desired assertion.
\end{proof}

Now, we are in a position to prove the main result in this section.
\begin{theorem}\label{th1-23} {\bf (General jumping kernel case)}\,\, Let $(M,d,\mu)$ be a
metric measure space, and $(D,\mathscr{F})$ be the
non-local regular Dirichlet form defined in \eqref{nondi}. Suppose that \eqref{e:ack} and \eqref{e:ack1} hold. Then, for any $p\in (1,2]$, $\mathscr{H}_{\widetilde \nabla}$
is bounded in $L^p(M,\mu)$, i.e., there exists a constant $c_p>0$ such that, for every $f\in L^p(M,\mu)$,
 $$\|\mathscr{H}_{\widetilde \nabla}f\|_{p}\le c_p\|f\|_{p}.$$
 \end{theorem}

\begin{proof} For each $n\ge 1$, denote by $(L_n, \mathscr{D}(L_n))$ and $(P_t^n)_{t\geq0}$ the generator and the semigroup associated with
the
Dirichlet form $(D^n,\F^n)$, respectively. According to \eqref{e:ack1} and Theorem \ref{th1}, for any $p\in (1,2]$, we can find a
constant $c_p>0$ such that for all $n\ge1$ and $f\in L^p(M,\mu)$,
$$\|\mathscr{H}_{\widetilde \nabla, n}f\|_p\le c_p\|f\|_p,$$
where $\mathscr{H}_{\widetilde \nabla, n}$ is defined by \eqref{eeefff} with $L_n$ in place of $L$ and \eqref{g-21} with the jumping
kernel $J_n(x,dy)=J_n(x,y)\,\mu(dy)$; more precisely,
 $$\mathscr{H}_{\widetilde \nabla, n}(f)(x)=\Big(\int_0^\infty\!\!\int_{\{y\in M: |P_t^nf|(x)\geq |P_t^nf|(y)\}}\big(P_t^nf(x)-P_t^nf(y)\big)^2J_n(x,dy)\,dt\Big)^{1/2}.$$
Note that, from the argument of Theorem \ref{th1}, the constant $c_p$ here is independent of $n\ge1$.

On the other hand, by Proposition \ref{P:mos} and Remark \ref{r:mos} (2), we know that for every $t>0$ and $f\in C_c(M)$,
$P_t^nf$ converges to $P_tf$ in $L^2(M,\mu)$ as $n\rightarrow\infty$. Hence, there is a subsequence $\{P_t^{n_k}f\}_{k\geq1}$
which converges to $P_tf$ $\mu$-a.e. as $k\rightarrow\infty$. By the definition of $\mathscr{H}_{\widetilde \nabla,n}(f)$ and the assumption \eqref{e:ack1}, we obtain that $\mathscr{H}_{\widetilde \nabla,n_k}(f)\rightarrow \mathscr{H}_{\widetilde \nabla}(f)$ $\mu$-a.e. as $k\rightarrow\infty$.
By using the Fatou lemma twice,
\begin{align*}&c_p\|f\|_p\\
&\ge \liminf_{n\to\infty} \int_M \mathscr{H}_{\widetilde \nabla, n}(f)^p(x)\,\mu(dx)\\
&=\liminf_{n\to\infty} \int_M \!\Big(\int_0^\infty\!\!\int_{\{y\in M: |P_t^nf|(x)\geq |P_t^nf|(y)\}}\big(P_t^nf(x)-P_t^nf(y)\big)^2J_n(x,dy)\,dt\Big)^{p/2}\,\mu(dx)\\
&\ge \int_M \!\Big(\liminf_{n\to\infty}\int_0^\infty\!\!\int_{\{y\in M: |P_t^nf|(x)\geq |P_t^nf|(y)\}}\big(P_t^nf(x)-P_t^nf(y)\big)^2J_n(x,dy)\,dt\Big)^{p/2}\,\mu(dx)\\
&\ge \int_M\Big(\int_{\{y\in M: |P_tf|(x)\geq |P_tf|(y)\}}\big(P_tf(x)-P_tf(y)\big)^2J(x,dy)\,dt\Big)^{p/2}\,\mu(dx) .\end{align*}
Hence, there is a constant $c_p>0$ such that
$$\|\mathscr{H}_{\widetilde \nabla}(f)\|_p\leq c_p\|f\|_p\quad\mbox{for any } f\in C_c(M).$$

The general case for $f\in L^p(M,\mu)$ is accomplished by approximation since $C_c(M)$ is dense in $L^p(M,\mu)$ for all $1\leq p<\infty$ and by Fatou's lemma.
 \end{proof}

\  \

It is easy to know that
\begin{equation}\label{bu2}\|\mathscr{H}_{\widetilde \nabla}f\|_{2}\leq\frac{\sqrt{2}}{2}\|f\|_{2}\quad\mbox{for all }0\le f\in L^2(M,\mu).\end{equation}
Indeed, letting $\{E_\lambda: 0\leq\lambda<\infty\}$ be the spectral representation of $L$, for any $h\in L^2(M,\mu)$, we have
$$D(h,h)=\int_0^\infty \lambda \,d\langle E_\lambda h,E_\lambda h\rangle$$  and
$$D(P_th, P_t h)=\int_{[0,\infty)} \lambda e^{-2\lambda t}\,d\langle E_\lambda h,E_\lambda h\rangle.$$
Then, for all $0\le f\in L^2(M,\mu)$, \begin{align*}\|\mathscr{H}_{\widetilde \nabla}f\|_{2}^2&=\int_{[0,\infty)} D(P_tf, P_t f)\,dt=\int_{[0,\infty)} \lambda e^{-2\lambda t}\,dt \int_{[0,\infty)} \,d \langle E_\lambda f, E_\lambda f\rangle\\
&=\frac{1}{2}\int_{(0,\infty)} \,d \langle E_\lambda f, E_\lambda f\rangle\leq\frac{1}{2}\|f\|_{2}^2.\end{align*}
However, for general $f\in L^2(M,\mu)$, we can not derive \eqref{bu2} as above. This shows that even
$L^2$ boundedness of $\mathscr{H}_{\widetilde \nabla}$ seems to be
non-trivial, which differs from the classic Littlewood--Paley theory in the case $p=2$.

In harmonic analysis, we are also interested in the so-called vertical Littlewood--Paley $\mathscr{G}$-function of the following form: for $f\in L^1(M,\mu)\cap L^\infty(M,\mu)$,
$$\mathscr{G}_{\widetilde \nabla}f(x)= \left(\int_0^\infty t |\widetilde \nabla e^{-t\sqrt{L}}f|^2(x)\,dt \right)^{1/2},\quad x\in M.$$
Inspired by the argument of \cite[Remark 1.3(ii)]{CDD}, we know that the function $\mathscr{G}_{\widetilde \nabla}f$ is  dominated pointwise by $\mathscr{H}_{\widetilde \nabla}f$.
Indeed, by using the fact
$$\int_0^\infty e^{-u} u^{1/2}\,du=\frac{\sqrt{\pi}}{2},$$
and applying the formula
$$e^{-t \sqrt{L}}=\frac{1}{\sqrt{\pi}}\int_0^\infty  e^{-t^2L/(4u)} e^{-u} u^{-1/2}\,du,\quad t\geq0,$$
we deduce from Jensen's inequality,  Fubini's theorem and the change-of-variables formula  that
\begin{align*}(\mathscr{G}_{\widetilde \nabla}f)^2(x)&=\int_0^\infty t|\widetilde \nabla e^{-t\sqrt{ L}} f|^2(x)\,dt\\
&=\frac{1}{\pi}\int_0^\infty t\Big(\int_0^\infty |\widetilde \nabla e^{-t^2L/(4u)}f|(x)e^{-u}u^{-1/2}\,du\Big)^2\,dt\\
&\le \frac{1}{\sqrt{\pi}}\int_0^\infty t \Big( \int_0^\infty |\widetilde \nabla e^{-t^2L/(4u)}f|^2(x)\,e^{-u}u^{-1/2}\,du\Big) \, dt\\
&=\frac{1}{\sqrt{\pi}}\int_0^\infty \Big(\int_0^\infty  t|\widetilde \nabla e^{-t^2L/(4u)}f|^2(x)\,dt\Big)e^{-u}u^{-1/2}\,du\\
&=\frac{2}{\sqrt{\pi}} \int_0^\infty e^{-u}u^{1/2}\,du\int_0^\infty  |\widetilde \nabla e^{-sL}f|^2(x)\,ds\\
&=(\mathscr{H}_{\widetilde \nabla}f)^2(x);
 \end{align*}
hence,
 $$(\mathscr{G}_{\widetilde \nabla}f)(x)\leq (\mathscr{H}_{\widetilde \nabla}f)(x).$$
In particular, the $L^p$ boundedness of $\mathscr{H}_{\widetilde \nabla}$ implies the $L^p$ boundedness of $\mathscr{G}_{\widetilde \nabla}$.

\ \

At the end of this section, we make some comments on Theorem \ref{th1-23} and its proof. In $\R^d$,
the boundedness of $\mathscr{H}_{\widetilde \nabla}$ in $L^p(\R^d,dx)$, $1<p\le 2$, for L\'evy process $X$  has been proved in \cite[Theorem 4.1 and Lemma 4.5]{BBL}, when the process $X$ satisfies the Hartman--Wintner condition. (Note that such condition implies that the process $X$ has a transition density function $p_t(x)$ with respect to the Lebesgue measure such that $\lim_{t\to\infty}p_t(0)=0$.)
The aforementioned approach  differs from ours, and it is based on the Hardy--Stein identity (see \cite[Theorem 3.2 and (3.5)]{BBL}). It seems that such identity  depends heavily on the characterization of L\'evy processes, and may not hold for general jump processes.
The authors mentioned in the introduction section of their paper  \cite{BBL} that --- \emph{The results should hold in a much more general setting, but the scope of the extension is unclear at this moment}.

\ \

As mentioned in Section \ref{section1}, it seems more natural to study the boundedness of the vertical Littlewood--Paley $\mathscr{H}$-function
defined in terms of $\nabla$, that is, the $L^p$ boundedness of the operator
$$\mathscr{H}_{\nabla} f(x)=\left(\int_0^\infty |\nabla e^{-tL}f|^2(x)\,dt \right)^{1/2},\quad x\in M.$$ By the same argument for \eqref{bu2}, it holds true that $\|\mathscr{H}_{\nabla}f\|_{2}\leq\|f\|_{2}/\sqrt{2}$ for all  $0\le f\in L^2(M,\mu).$  However, \cite[Example 2]{BBL}, which is inspired by \cite[Page 165-166]{Ben}, shows that in general settings the operator $\mathscr{H}_{\nabla}$
may fail to be bounded on $L^p(M,\mu)$ if $1<p<2$. Thus, $\mathscr{H}_{\nabla}$ and $\mathscr{H}_{\widetilde\nabla}$ differ considerably. Another point is that, if we define $\mathscr{H}_{\widetilde\nabla_*}$ as $\mathscr{H}_{\widetilde\nabla}$ by using the modified gradient $|\widetilde\nabla f|_*$ defined by \eqref{g-2} instead of $|\widetilde\nabla f|$ defined by \eqref{g-21},  one may wonder whether $\mathscr{H}_{\widetilde\nabla_*}$ is bounded on $L^p(M,\mu)$ for $p\in (1,2]$ or not. The answer is no! Indeed, suppose that $\mathscr{H}_{\widetilde\nabla_*}$ is bounded on $L^p(M,\mu)$ for $p\in (1,2]$. Then, a simple change of $f$ into $-f$ will give us the boundedness of the operator $\mathscr{H}_{\nabla}-\mathscr{H}_{\widetilde\nabla_*}$, which forces that $\mathscr{H}_{\nabla}$ is bounded too. However, this is a contradiction.

In comparison with the discrete setting in \cite{Nick}, one  
key point of Theorem \ref{th1-23} is  
the boundedness of $\mathscr{H}_{\widetilde\nabla}$ in $L^p(M,\mu)$ for $p\in (1,2]$, not only in the particular cone $L^p_+(M,\mu):=\{f\geq0: f\in L^p(M,\mu)\}$. The reason is that the operator $\mathscr{H}_{\widetilde\nabla}$ does not enjoy the sublinear property.  Also due to this, we need to define the pseudo-gradient $\Gamma_p$ for suitable signed function $f$; see \eqref{e:ger0}. By some simple calculation, we can deduce that, for any measurable function $f$ on $M$,
\begin{align*}\label{absolute-control}
\I_{\{z\in M:|f(x)|\geq|f(z)|\}}(y)\big(f(x)-f(y)\big)^2
\leq & 4\I_{\{z\in M:f^+(x)\geq f^+(z)\}}(y)\big(f^+(x)-f^+(y)\big)^2\\
& + 4\I_{\{z\in M:f^-(x)\geq f^-(z)\}}(y)\big(f^-(x)-f^-(y)\big)^2
\end{align*}holds for all $x,y\in M$. So, it should be a feasible way to prove the boundedness of $\mathscr{H}_{\widetilde\nabla}$ in $L^p(M,\mu)$ from its boundedness in $L^p_+(M,\mu)$.

\section{Littlewood--Paley--Stein estimates for $2\le p<\infty$}
Recall that $(M,d,\mu)$ is a metric measure space,  $(D,\F)$ given by \eqref{nondi} is  a regular Dirichlet form, $(L,\mathscr{D}(L))$ and $(P_t)_{t\ge0}:=(e^{-tL})_{t\ge0}$ are the corresponding $L^2$-generator and $L^2$-semigroup, respectively.  Associated with the regular Dirichlet form $(D,\mathscr{F})$ on $L^2(M,\mu)$ there is a symmetric Hunt process $X=\{X_t,t\ge0, \Pp^x, x\in M\backslash\mathscr{N}\}$. Here $\mathscr{N}$ is a properly exceptional set for $(D,\mathscr{F})$ in the sense that $\mu(\mathscr{N})=0$ and $\Pp^x(X_t\in \mathscr{N}\textrm{ for some }t>0)=0$ for all $x\in M\backslash\mathscr{N}.$ See \cite{CF, FOT} for more details.

\ \

Throughout this section, we make the following assumptions.

\begin{itemize}
\item[$(A1)$] For any $ f\in C_c(M)$, the function $(t,x)\mapsto P_tf(x)$ is continuous on  $(0,\infty)\times M$.

\item[$(A2)$]  The process $X$ has a transition density function $p_t(x,y)$ with respect to the reference measure $\mu$, i.e.,
for any $t>0$, $x\in M\backslash\mathscr{N}$ and any Borel set $B\subset M$,
$$\Pp^x(X_t\in B)=\int_B p_t(x,y)\,\mu(dy).$$

\item[$(A3)$] The process $X$ is conservative,
i.e., for any $t>0$ and $x\in M\backslash\mathscr{N}$,
$$\int_M p_t(x,y)\,\mu(dy)=1.$$

\item[$(A4)$] There exist a $\sigma$-finite measure space $(U,\mathscr{U},\nu)$ and a function $k:$ $M\times U\rightarrow M$ such that for any Borel set $B\subset M$ and $\mu$-a.e. $x\in M$,
\begin{equation}\label{ss3}\nu\left\{z\in U: k(x,z)\in B\right\}=J(x,B).\end{equation}
\end{itemize}

We make some comments on assumptions above. Firstly, according to \cite[Chapter 1, Lemma 1.4, p.\ 5]{BSW}, assumption $(A1)$ holds if the semigroup $(P_t)_{t\ge0}$ enjoys the $C_\infty$-Feller property; that is, for any $t>0$ and $f\in C_\infty(M)$, $P_tf\in C_\infty(M)$, and $\lim_{t\to 0}\|P_tf-f\|_\infty =0$, where $C_\infty(M)$ denotes the set of continuous functions which varnish at infinity. Secondly, there are already a few works on the conservativeness of processes generated by non-local Dirichlet forms on metric
measure
spaces;  see e.g. \cite{MUW} and the references therein.  Thirdly, assumption $(A4)$ is our technique condition. When $M=\R^d$, one can take $U=\R^n$ and $\nu(dz)= |z|^{-n-1}\,dz$ with $n\ge2$, and find a measurable function $k:\R^d\times \R^n\to \R^d$ such that \eqref{ss3} is satisfied; see \cite[Chapeter 3, Theorem 3.2.5]{Sto}. Hence, assumption $(A4)$ always holds in the Euclidean space. As a general result on construction of the coefficient $k(x,z)$ in \eqref{ss3}, we refer to El Karoui
and Lepeltier \cite{KL}, where they constructed $k(x,z)$ under the condition that $U$ is a Lusin
space and $\nu$ is a $\sigma$-finite diffusive measure on $U$ with infinite total mass.

\ \

For fixed $f\in C_c(M)$ and $T>0$, let
$$H_t=P_{T-t}f(X_t)-P_Tf(X_0),\quad 0\le t\le T.$$ Denote by $(\mathscr{F}_t)_{t\ge0}$ the natural filtration of the process $X$. Then, we have
\begin{lemma}\label{lemm31} Under the assumption $(A1)$, $\{H_t, \mathscr{F}_t\}_{0\le t\le T}$ defined above is a
martingale starting at $0$, and
for any $0\le t\le T$,
\begin{equation}
\label{eek01}[H]_t= \int_0^t\int_M\left( P_{T-s} f(y)- P_{T-s}f(X_{s-})\right)^2\,J(X_{s-},dy)\,ds,
\end{equation}
where $[H]_t$ is the quadratic variation of $H_t$.
\end{lemma}

The statement above for symmetric L\'evy processes can be obtained directly via the It\^{o} formula; see \cite[Section 4]{BBL}.  However, since the It\^{o} formula is not available in the present setting, we will use a different approach.

\begin{proof}[Proof of Lemma $\ref{lemm31}$] For any $0\le s\le t\le T$, by the Markov property, $$P_{T-t}f(X_t)=\Ee^{X_t} f(X_{T-t})=\Ee (f(X_T)|\mathscr{F}_t),$$ and so
\begin{align*}\Ee(H_t|\mathscr{F}_s)=&\Ee(P_{T-t}f(X_t)-P_Tf(X_0)|\mathscr{F}_s)=\Ee(P_{T-t}f(X_t)|\mathscr{F}_s)-P_Tf(X_0)\\
=&\Ee(\Ee (f(X_T)|\mathscr{F}_t)|\mathscr{F}_s)-P_Tf(X_0)=\Ee(f(X_T)|\mathscr{F}_s)-P_Tf(X_0)\\
=&P_{T-s} f(X_s)-P_T f(X_0)=H_s.\end{align*} This proves that $\{H_t, \mathscr{F}_t\}_{0\le t\le T}$ is a martingale.

For any $x\in M$ and $0\le t\le T$, we have \begin{align*}\Ee^x(H_t^2)=&
\Ee^x\big(P_{T-t}f(X_t)-P_Tf(X_0)\big)^2\\
=&\Ee^x(P_{T-t}f)^2(X_t)-2P_Tf(x)\Ee^xP_{T-t}f(X_t)+(P_Tf)^2(x)\\
=& P_t(P_{T-t}f)^2(x)- (P_Tf)^2(x).\end{align*}
Then, for any $x\in M$  and $0\le s\le t\le T$,
\begin{equation}\label{e:fffee}\begin{split}
 \Ee^x(H_t^2-H_s^2)=& P_t(P_{T-t}f)^2(x)- P_s(P_{T-s}f)^2(x)\\
=& \int_s^t \frac{ d(P_u(P_{T-u} f)^2)(x)}{du}\,du\\
=&\int_s^t\Big(-LP_u (P_{T-u}f)^2(x)+P_u(2P_{T-u}f \cdot L P_{T-u}f)(x)\Big)\,du \\
 =& \int_s^t P_u\Big(-L(P_{T-u} f)^2+2P_{T-u} f\cdot LP_{T-u}f\Big)(x)\,du\\
 =&2\int_s^t P_u\Gamma (P_{T-u}f )(x)\,du\\
 =&2\Ee^x \Big( \int_s^t \Gamma (P_{T-u}f)(X_u)\,du\Big),\end{split}\end{equation}
 where in the penultimate equality above we used the fact that
 $$\Gamma(f)= \frac{1}{2}\left(2fLf - L(f^2)\right).$$
See e.g. \cite[Theorem (3.7)]{CKS}.
\eqref{e:fffee} together with the Markov property in turn yields that
\begin{equation}\label{eek00}\Big\{H_t^2 - 2\int_0^t \Gamma (P_{T-u}f)(X_u)\,du, \mathscr{F}_t\Big\}_{\ 0\le t\le T}\end{equation}
is a martingale. Denote by $\langle H\rangle_t$ and $[H]_t$ the predicable
quadratic variation and the quadratic variation of $H_t$, respectively.
  Thus, according to \eqref{eek00} and \cite[Chapter 4, Theorem 4.2, p.\ 38]{JS},
 $$\langle H\rangle _t=2\int_0^t \Gamma (P_{T-u}f)(X_u)\,du, \quad 0\le t\le T.$$

Furthermore, under assumption $(A1)$ and by the fact that $t\mapsto X_t$ is quasi-left-continuous (since $X$ is a Hunt process enjoying the strong Markov property), $\{H_t,\mathscr{F}_t\}_{0\le t\le T}$ is a martingale which has a continuous version, see again \cite[Chapter 4, Theorem 4.2, p.\ 38]{JS}.
Hence, by \cite[Chapter 2,  Definition 2.25, p.\ 22; Chapter 4, Theorem 4.52, p.\ 55]{JS}, for any $0\le t\le T$,
\begin{align*}[H]_t=&\langle H\rangle _t=2\int_0^t \Gamma (P_{T-u}f)(X_u)\,du\\
=&\int_0^t\int_M\big( P_{T-u} f(y)- P_{T-u}f(X_{u-})\big)^2\,J(X_{u-},dy)\,du.\end{align*}
The proof is complete.
\end{proof}

Next, we will make use of the space-time parabolic martingale $\{H_t, \F_t\}_{0\le t\le T}$ defined above with $T\in(0,\infty]$ to prove the boundedness of Littlewood--Paley functions in $L^p(M,\mu)$ for $2\le p<\infty$. We mainly follow the approach of \cite[Section 4]{BBL} with
necessary
modifications.

First, we note that, under assumption $(A3)$, one can rewrite \eqref{eek01} for the quadratic variation $[H]_t$ of the martingale $H_t$ as
\begin{equation}\label{eek02}
[H]_t= \int_0^t\int_M\big( P_{T-s} f(k(X_{s-},y))- P_{T-s}f(X_{s-})\big)^2\,\nu(dy)\,ds,\quad 0\le t\le T.
\end{equation}

Second, we need to define the Littlewood--Paley function $G$, which can be regard as the conditional expectation of the quadratic variation $[H]_t$.
For $f\in L^1(M,\mu)\cap L^\infty(M,\mu)$, define
$$Gf(x)=\left(\int_0^\infty\!\! \int_M\!\int_M |P_t f(k(z,y))-P_tf(z)|^2p_t(x,z)\,\mu(dz)\,\nu(dy)\,dt\right)^{1/2},$$ and
$$G_{T}f(x)=\left(\int_0^T \!\!\int_M\!\int_M |P_t f(k(z,y))-P_tf(z)|^2p_t(x,z)\,\mu(dz)\,\nu(dy)\,dt \right)^{1/2}.$$
Clearly, $\lim_{T\to \infty}G_{T} f(x)=Gf(x)$ for all $x\in M$.
\begin{lemma}\label{lem-1} Under assumptions $(A1)$, $(A2)$,  $(A3)$ and $(A4)$, for any $f\in C_c(M)$, $x\in M$ and $T>0$,
$$(G_{T} f)^2(x)=\int_M \Ee^z ([H]_T|X_T=x)p_T(z,x)\,\mu(dz).$$
\end{lemma}

\begin{proof}The proof is almost the same
as that of \cite[(4.5)]{BBL}, and we present it here for the sake of completeness. By \eqref{eek02}, we have
 \begin{align*}
&\int_M\Ee^z ([H]_T|X_T=x)p_T(z,x)\,\mu(dz)\\
&=\int_M\Ee^z\left(   \int_0^T\int_M\big( P_{T-s} f(k(X_{s-},y))- P_{T-s}f(X_{s-})\big)^2\,\nu(dy)\,ds\Big| X_T=x\right)\\
&\qquad\qquad\qquad\times p_T(z,x)\,\mu(dz)\\
&=\int_M \bigg(\int_0^T\int_M\frac{p_s(z,w)p_{T-s}(w,x)}{p_T(z,x)} \\
&\qquad\qquad\qquad\times \int_M |P_{T-s} f(k(w, y))-P_{T-s}f(w)|^2\,\nu(dy)\,\mu(dw)\,ds\bigg)p_T(z,x)\,\mu(dz)\\
&=\int_0^T\int_M p_{T-s}(w,x) \int_M|P_{T-s} f(k(w, y))-P_{T-s}f(w)|^2\,\nu(dy)\,\mu(dw)\,ds\\
&=\int_0^T \int_M  \int_M|P_{T-s} f(k(w, y))-P_{T-s}f(w)|^2p_{T-s}(x,w)\,\mu(dw)\,\nu(dy)\,ds\\
&=(G_{T}f)^2(x), \end{align*}
where
in the third equality we used the fact that
$$\int_M p_s(z,w)\,\mu(dz)= \int_M p_s(w,z)\,\mu(dz)=1\quad\mbox{for any }w\in M\backslash\mathscr{N}, $$
due to the symmetry
of $p_t(x,y)$ and the conservativeness of the process $X$, and in the
fourth equality we used symmetry again.
\end{proof}

\begin{lemma}\label{lem-2} Under assumptions $(A1)$, $(A2)$ and $(A4)$, for any $f\in C_c(M)$ and $x\in M$,
$$(\mathscr{H}_\nabla f)(x)\le Gf(x).$$ \end{lemma}
\begin{proof}
For any $f\in C_c(M)$ and $x\in M$, using the Cauchy--Schwarz inequality, the  property of the semigroup $(P_t)_{t\ge0}$ and  \eqref{ss3}, we get
\begin{align*}(G f)^2(x)=&\int_0^\infty \!\!\int_M\! \int_M|P_{t}f(k(z,y))-P_{t}f(z)|^2p_{t}(x,z)\,\mu(dz)\,\nu(dy)\,dt\\
\ge&\int_0^\infty \int_M\! \left(\int_M |P_{t}f(k(z,y))-P_{t}f(z)|p_{t}(x,z)\,\mu(dz)\right)^2\,\nu(dy)\,dt\\
=&\int_0^\infty \int_M\! \Big( P_t |P_{t}f(k(\cdot,y))-P_{t}f(\cdot)| (x)\Big)^2\,\nu(dy)\,dt\\
\ge &\int_0^\infty \int_M\! \big(|P_{2t}f(k(\cdot,y))-P_{2t}f(\cdot)| (x)\big)^2\,\nu(dy)\,dt\\
=&\int_0^\infty \left(\int_M\! \big(|P_{2t}f(z)-P_{2t}f(\cdot)|\big)^2\,J(\cdot, dz)\right) (x)\,dt\\
=&2\int_0^\infty \Gamma(P_{2t}f
) (x)\,dt=\int_0^\infty \Gamma(P_{t}f ) (x)\,dt\\
= &\int_0^\infty|\nabla P_t f|^2(x)\,dt=(\mathscr{H}_\nabla f)^2(x). \end{align*}
This proves the desired assertion.
\end{proof}

Now, we are in a position to present the main result in this section.
\begin{theorem}\label{thp}
Let $(M,d,\mu)$ be a metric measure space, and $(D,\mathscr{F})$ be the non-local regular Dirichlet form defined in \eqref{nondi}.  Suppose that $(A1)$, $(A2)$, $(A3)$ and $(A4)$ hold. Then, for any $p\in [2,\infty)$,
 $\mathscr{H}_\nabla$ is bounded in $L^p(M,\mu)$, i.e., there exists a constant $C_p>0$ such that, for every $f\in L^p(M,\mu)$,
 $${\|\mathscr{H}_\nabla f\|_{p}\le C_p\|f\|_{p}.}$$
 \end{theorem}
 \begin{proof}
 Let $p\ge 2$. For any $f\in C_c(M)$,
 by Lemma \ref{lem-1} and
 by applying
 Jensen's inequality twice, we get
 \begin{align*}\int_{M} (G_T f)^p(x)\,\mu(dx)&\le \int_{M}\!\!\int_{M} \left(\Ee^z([H]_T|X_T=x)\right)^{p/2} p_T(z,x)\,\mu(dz)\,\mu(dx)\\
 &\le \int_{M}\!\int_{M} \Ee^z([H]_T^{p/2}|X_T=x)p_T(z,x)\,\mu(dz)\,\mu(dx)\\
&=\int_{M} \Ee^z([H]_T^{p/2})\,\mu(dz).
\end{align*}
 By the Burkholder--Davis--Gundy inequality (see e.g. \cite[Page 234]{Ap}), we have
 \begin{align*}
 \Ee^z([H]_T^{p/2})&\le C_p'\Ee^z|H_T|^p=C_p'\Ee^z\big(|f(X_T)-P_Tf(X_0)|^p\big)\\
 &\le 2^pC_p' P_T|f|^p(z),
 \end{align*}
 where in the last inequality we have used the elementary fact that
 $$(a+b)^p\le 2^{p-1}(a^p+b^p),$$
 for all $a,b\ge0$ and $p\geq1$. Thus, by the contraction property of the semigroup $(P_t)_{t\ge0}$ on $L^p(M,\mu)$, we arrive at
 $$\int_{M} (G_T f)^p(x)\,\mu(dx)\le 2^pC_p'\int_{M} P_T|f|^p(x)\,\mu(dx)\le 2^pC_p'\|f\|_{p}^p,$$ which along with the monotone convergence theorem yields that
 $$ \int_{M} (Gf)^p(x)\,\mu(dx)\le 2^pC_p'\|f\|_{p}^p.$$
Combining this with Lemma \ref{lem-2}, we obtain that
$$\|\mathscr{H}_\nabla f\|_p\leq C_p\|f\|_p\quad\mbox{for every }f\in C_c(M).$$
For every $f\in L^p(M,\mu)$, since $C_c(M)$ is dense in $L^p(M,\mu)$ for all $1\leq p<\infty$,  we may choose a sequence $\{f_n\}_{n\geq1}$ from $C_c(M)$ such that $f_n$ converges to $f$ in $L^p(M,\mu)$ as $n\rightarrow\infty$. Then, applying Fatou's lemma, we obtain the desired conclusion.
 \end{proof}

Finally, we present the proof of Example \ref{exm2}.
\begin{proof}[Proof of Example $\ref{exm2}$] According to comments after assumptions in the beginning of this section, we only need to
verify assumptions $(A1)$---$(A3)$.  According to \cite[Proposition 3.3]{BKK}, the associated resolvent of $(D,\mathscr{F})$ is H\"{o}lder continuous, which entails that the $L^2$-semigroup $(P_t)_{t\ge0}$ generated by the Dirichlet form $(D,\mathscr{F})$ is a $C_\infty$-Feller semigroup; see \cite[Proposition 4.3 and Corollary 6.4]{RU}. Thus, $(A1)$ holds. From \eqref{ed4}, one can easily
deduce that
the Nash type inequality holds for $(D,\F)$, which implies that the associated process $X$ enjoys the transition density function
(see \cite{BKK} for more details); that is, $(A2)$ is also satisfied. $(A3)$ immediately follows from \cite[Theorem 1.1]{MUW}. Therefore, we can obtain the desired assertion from Theorem \ref{thp}.\end{proof}

\subsection*{Acknowledgment}
\hskip\parindent\!\!\!
The second author would like to thank
Professor Kazuhiro Kuwae for very
helpful comments on the Mosco convergence of Dirichlet forms and the proof of Lemma \ref{lemm31}. The research of Huaiqian Li is supported by National Natural Science Foundation of China (Nos.\ 11401403, 11571347).
The research of Jian Wang is supported by National
Natural Science Foundation of China (No.\ 11522106), the Fok Ying Tung
Education Foundation (No.\ 151002) and the Program for Probability and Statistics: Theory and Application (No.\ IRTL1704).

\end{document}